\documentclass[reqno]{amsart}
\usepackage{graphicx}
\usepackage{amssymb,amsmath}
\usepackage{amsthm}
\usepackage{amssymb,amsmath}
\usepackage{amsthm}
\usepackage{amsfonts}
\usepackage[mathscr,mathcal]{eucal}

\newtheorem{thm}{Theorem}[section]
\newtheorem{cor}[thm]{Corollary}
\newtheorem{lem}[thm]{Lemma}

\theoremstyle{definition}

\theoremstyle{remark}

\numberwithin{equation}{section}

\newcommand{\Z}{\mathbb{Z}}

\newcommand{\Aut}{\mathrm{Aut}}

\newcommand{\Inn}{\mathrm{Inn}}
\newcommand{\ex}{\mathrm{exp}}

\newcommand{\Fr}{\Phi}

\newcommand{\ce}{\mathcal{C}}

\newcommand {\la}{\langle}
\newcommand {\ra}{\rangle}

\begin{document}

\title[Noninner automorphisms of order $p$]
{Noninner automorphisms of order $p$ in finite $p$-groups of coclass 2,  when $p>2$}%
\author[ S. Fouladi and R. Orfi ]{ S. Fouladi and R. orfi*}%

 \address{Faculty of Mathematical Sciences and Computer, Kharazmi University,
50 Taleghani Ave., Tehran 1561836314, Iran}
 \email{ s\_fouladi@khu.ac.ir}
 \address{ Department of Mathematics,  Faculty of Science, Arak University, Arak 38156-8-8349, Iran.}
 \email{ r-orfi@araku.ac.ir}

\subjclass[2000]{20D45, 20D15}%
\keywords{automorphisms of $p$-groups, $p$-groups of coclass 2, noninner automorphisms, derivations}%
\thanks {*Corresponding author: r-orfi@araku.ac.ir}
\date{\today}
\begin{abstract}It is shown that if $G$ is a finite $p$-group of coclass 2 with $p>2$, then 
$G$ has a noninner automorphism of order $p$.
\end{abstract}

\maketitle
\section{\bf{Introduction}}

\vspace*{0.4cm} 
Let $G$ be a finite non-abelian $p$-group. A longstanding conjecture asserts that $G$ possesses at least one noninner automorphism of order $p$ (see \cite[Problem 4.13]{Mazur}). This is a sharpened version of a celebrated theorem of Gasch\"{u}tz \cite{Gasch} which states that finite non-abelian $p$-groups have noninner automorphisms of $p$-power order. By a result of Deaconescu and Silberberg  
\cite{Deacon} if a $p$-group $G$ satisfies $\ce_G(Z(\Fr(G)))\neq \Fr(G)$, then $G$ admits a noninner automorphism of order $p$ leaving $\Fr(G)$ elementwise fixed. However the conjecture is still open, various attempts have been made to find noninner automorphisms of order $p$ in 
some classes of finite $p$-groups (see 
 \cite{Abdol3},  \cite{Bodna}, \cite{Deacon}, \cite{Jamal}, \cite{Schmid} and \cite{Shabani}). In particular the conjecture has been proved for finite $p$-groups of class 2, class 3 and of maximal class (see \cite{Lieb}, 
\cite{Abdol1}, \cite{Abdol2} and \cite[Corollary 2.7]{Shabani}). In this paper we show the validity of the conjecture 
when $G$ is a finite $p$-group of coclass 2 with $p>2$  (see Theorem \ref{7}). By the nilpotency coclass 
of a $p$-group of order $p^n$, we mean the number $n-c$, where $c$ is the nilpotency class of $G$. 

Throughout this paper the following notation is
used. $\Aut^{N}(G)$ denotes the group of all
automorphisms of $G$ normalizing $N$ and centralizing $G/N$, and $\Aut_{N}(G)$ denotes the group of all
automorphisms of $G$ centralizing $N$. Moreover $\Aut_{N}^{N}(G)=\Aut^{N}(G) \cap \Aut_{N}(G)$. 
All central automorphisms of $G$ is denoted by $\Aut_c(G)$. the terms of the upper central 
series of $G$ is denoted by $Z_i=Z_{i}(G)$.  The group of all derivations from $G/N$ to $Z(N)$ is denoted by $Z^{1}(G/N,Z(N))$, where $N$ is a normal subgroup of a group $G$ and $G/N$ acts on $Z(N)$ as 
$a^{Ng}=a^g$ for all $a\in Z(N)$ and $g\in G$. We use the notation $x\equiv y$ $\pmod H$
to indicate that $Hx=Hy$, where $H$ is a subgroup of a group $G$ and $x, y \in G$. 
The minimal number of generators of $G$ is denoted by $d(G)$ and $C_n$ is the cyclic group of order
$n$. All unexplained notation is standard. Also a nonabelian group $G$ that
has no nontrivial abelian direct factor is said to be purely
nonabelian.

\vspace*{0.4cm}

\section{\bf{The main result}}

\vspace*{0.4cm} In this section we aim to prove that if $G$ is a $p$-group of order $p^n$ $ (p>2)$ and
 coclass $2$, then $G$ has a noninner automorphism of order $p$. First we may assume that
$n\geq 7$ by \cite{Bodna} and \textsf{GAP}\cite{GAP} when $p=3$. Moreover we have the
following upper central series for $G$:
$$1<Z_1(G)<Z_2(G)< \dots < Z_{n-3}(G)<G$$ which indicates that $p^{n-3} \leq|Z_{n-3}(G)|\leq p^{n-2}$,
$p \leq|Z(G)|\leq p^2$ and $p^2 \leq|Z_2(G)|\leq p^3$. Now by \cite[Theorem (2)]{Shabani}, if $Z_2(G)/Z(G)$ is cyclic, then $G$ has a noninner automorphism of order $p$. Therefore we 
may assume that $|Z(G)|=p$, $|Z_2(G)|=p^3$ and $Z_2(G)/Z(G)\cong C_p\times C_p$. Also by \cite[Theorem (3)]{Shabani}, we deduce that $Z_2(G)\leq Z(\Fr(G))$ and $d(G)=2$ since in other cases
$G$ has a noninner automorphism of order $p$. Now by the above observation we state the following 
remark and we use the assumption and notation of it throughout the paper.

\vspace{0.5cm}
\noindent{\bf Remark. }
 Assume that  $G$ is  a group of order $p^n$ $(n\geq7, p>2)$ and 
coclass 2 with the following properties:
\begin{itemize}
\item[(a)] $|Z(G)|=p$ and $Z_2(G)/Z(G)\cong C_p\times C_p$,
\item[(b)] $Z_2(G)\leq Z(\Fr(G))$ and $d(G)=2$.
\end{itemize}  

\begin{lem}\label{1} We have 
\begin{itemize}
\item[(i)] $G$ is purely non-abelian,
\item[(ii)] $|Z_i(G)|=p^{i+1}$ for $2 \leq i \leq n-3$ and $Z_{n-3}(G)=\Fr(G)$,
\item[(iii)] $\exp(G/Z_{n-4})=p$,
\item[(iv)] $|\Aut_c(G)|=p^2$ and  $\Aut_c(G)\leq \Inn(G)$,
\item[(v)] there exists a normal subgroup $N$ of $G$ such that $N<Z_2(G)$,  
$N\cong C_p\times C_p$ and $\ce_G(N)$ is a maximal subgroup of $G$.
\end{itemize} 
\end{lem}

\begin{proof}
(i) and (ii) are obvious by Remark.\\
(iii) We set $G_{1}=G/Z_2$ and $G_{2}=G_{1}/\Gamma_{4}(G_1)$. Then
$G_1$ and $G_2$ are both of maximal class having order $p^{n-3}$
and $p^4$, respectively. Since $p+1 \geq 4$ by \break [\ref{Bl},
Theorem 3.2], $\ex(G_2/\Gamma_{3}(G_2))=p.$ However, since
$\Gamma_{3}(G_2)=\Gamma_{3}(G_1)/\Gamma_{4}(G_1),$
$$G_2/\Gamma_{3}(G_2)\cong G_1/\Gamma_{3}(G_1)=(G/Z_2)/(Z_{n-4}/Z_2),$$
completing the proof.\\
(iv) This follows  from (i), \cite[Theorem 1]{Ad} and the fact that  $\Aut_c(G)\cap \Inn(G)=Z(\Inn(G)).$ \\
(v) By the Remark, we see that $Z_2(G)$ is a non-cyclic abelian group of order $p^3$. If
$ Z_2(G) \cong C_p\times C_{p}\times C_{p}$  then we may choose $N$ such that $N/Z(G)$ is a subgroup
of order $p$ in 
$Z_2(G)/Z(G)$  and we set $N=\Omega_1(Z_2(G))$ if  $Z_2(G)\cong C_p^{2}\times C_{p}$. Moreover $G/\ce_{G}(N)\hookrightarrow GL(2,p)$, which completes the proof.
\end{proof}

\begin{lem}\label{5}
Let $b \in G\setminus C_{G}(N)$, $a \in C_{G}(N)\setminus\Phi(G)$ and $n \in N\setminus Z(G)$,
where $N$ is defined in Lemma $\ref{1}(v)$. 
Then 
\begin{itemize}
\item[(i)] $G=\la a, b \ra$,
\item[(ii)] $[a^r, b^s]\equiv [a,b]^{rs}$ $\pmod {Z_{n-4}}$, where $r$ and $s$ are integers,
\item[(iii)] the map $\alpha$ defined by $\alpha(Nfa^jb^i)=n^i[n,b]^{i(i-1)/2}$ is a derivation 
from $G/N$ to $N$, where $f\in \Fr(G)$ and $ i, j \in \Z$, 
\item[(iv)] the map $\beta$ defined by $\beta(Nx[a,b]^ta^jb^i)=n^j[n,b]^{ij+t}$ is a derivation 
from $G/N$ to $N$, where $x\in Z_{n-4}$ and $i, j, t \in \Z$.
\end{itemize} 
\end{lem}

\begin{proof}
(i) This is clear.\\
(ii) First assume that $r$ and $s$ are positive. By using induction on $r$ we see that 
$[a^r, b]\equiv [a,b]^{r}$ $\pmod {Z_{n-4}}$ 
since $[a^{r+1}, b]= [a^r, [a,b]^{-1}][a,b][a^r,b]$ and $[a^r, [a,b]^{-1}]\in Z_{n-4}$. Hence by using induction on $s$, we have
$[a^r, b^s]\equiv [a,b]^{rs}$ $\pmod {Z_{n-4}}$. The rest follows from 
$[a^{-r}, b^{-s}]=[b^{-s}a^{-r}, [a^r,b^s]^{-1}][a^r,b^s]$. \\
(iii) Since $|G/\ce_G(N)|=|\ce_G(N)/\Fr(G)|=p$, any element of $G$ is written as $fa^jb^i$, 
where $f\in \Fr(G)$ and $i, j \in \Z $. First we prove that $\alpha$ is well defined.  To see this let 
$g_1=f_1a^{j_1}b^{i_1}$ and $g_2=f_2a^{j_2}b^{i_2}$. If $Ng_1=Ng_2$, then $\ce_G(N)g_1=\ce_G(N)g_2$
and so $b^{i_2-i_1}\in \ce_G(N)$ which implies that $i_2=i_1+kp$ for some $k \in \Z$. Therefore 
$\alpha(Ng_2)=\alpha(Ng_1)$ since  $|n|=|[n,b]|=p$  and $p$ is odd.
 Now we have
${\alpha(Ng_1)}^{g_2}\alpha(Ng_2)=(n^{i_1})^{b^{i_2}}n^{i_2}[n,b]^{\frac{i_1(i_1-1)+i_2(i_2-1)}{2}}$ and $(n^{i_1})^{b^{i_2}}=n^{i_1}[n,b]^{i_1i_2}$ since $f_2a^{j_2} \in \ce_G(N)$ and $[n,b] \in Z(G)$. Hence ${\alpha(Ng_1)}^{g_2}\alpha(Ng_2)=n^{i_1+i_2}[n,b]^{\frac{(i_1+i_2-1)(i_1+i_2)}{2}}$. Moreover
$g_1g_2\equiv a^{j_1+j_2}b^{i_1+i_2}$ $\pmod {\Fr(G)}$ which completes the proof.\\
(iv) Since $|\Fr(G)/Z_{n-4}|=p$ and $[a, b] \in \Phi(G)\setminus Z_{n-4}$, any element of $G$ is determined by $x[a,b]^t a^jb^i$, where
$x \in Z_{n-4}$ and $ i,j,t \in \Z.$  First we prove that $\beta$ is well defined. To see this let  
$g_1=x_1[a,b]^{t_1}a^{j_1}b^{i_1}$ and $g_2=x_2[a,b]^{t_2}a^{j_2}b^{i_2}$. If $Ng_1=Ng_2$, then $b^{i_2-i_1}\in \ce_G(N)$ and so 
$i_2=i_1+kp$ for some $k \in \Z$. This implies that $j_2=j_1+lp$ for some $l \in \Z$ since  $\Phi(G)g_1=\Phi(G)g_2$. Therefore we see that  $t_2=t_1+up$ for some $u \in \Z$ by the fact that 
$Z_{n-4}g_1=Z_{n-4}g_2$   and Lemma \ref{1}(iii). Hence  $\beta$ is well defined.
Now we have 
${\beta(Ng_1)}^{g_2}\beta(Ng_2)=n^{j_1+j_2}[n,b]^{i_{1}j_1+t_1+i_{2}j_2+t_2+j_{1}i_2}$ since
 $x_2[a,b]^{t_2}a^{j_2} \in \ce_G(N)$ and $[n,b] \in Z(G)$. 
Moreover $Z(G/Z_{n-4})=\Fr(G)/Z_{n-4}$ by Lemma \ref{1}(ii), which yields that $g_1g_2\equiv [a,b]^{t_1+t_2}a^{j_1}b^{i_1}a^{j_2}b^{i_2}$ $\pmod {Z_{n-4}}$. Moreover
$a^{j_1}b^{i_1}a^{j_2}b^{i_2}=a^{j_1+j_2}[a^{j_2},b^{-i_1}]b^{i_1+i_2}$. Therefore
$g_1g_2\equiv [a,b]^{t_1+t_2-i_1j_2}a^{j_1+j_2}b^{i_1+i_2}$ $\pmod {Z_{n-4}}$ by (ii). Consequently 
$\beta(Ng_1Ng_2)={\beta(Ng_1)}^{g_2}\beta(Ng_2)$.
\end{proof}

We use the following theorem to complete our proof.

\begin{thm}\label{4}
Suppose that $N$ is a normal subgroup of a group $G$. Then there is a natural isomorphism
$\varphi : Z^1(G/N, Z(N))\rightarrow \Aut_{N}^N(G)$ given by $g^{\varphi(\gamma)}=g\gamma(Ng)$ for 
$g \in G$ and $\gamma \in Z^1(G/N, Z(N))$.
\end{thm}

\begin{proof}
See for example \cite[Result 1.1]{Schmid}.
\end{proof}

\begin{cor}\label{6}
By the assumption of Lemma $\ref{5}$, the maps $\alpha^*$ and $\beta^*$ defined by $a^{\alpha^*}=a$, $b^{\alpha^*}=bn$ and
$a^{\beta^*}=an$, $b^{\beta^*}=b$ are noncentral automorphisms of order $p$ lying in $\Aut_{N}^N(G)$.
\end{cor}

\begin{proof}
This is obvious by Lemma \ref{5} and Theorem \ref{4}.
\end{proof}

Now we deduce our main theorem: 

\begin{thm}\label{7}
Let $G$ be a finite $p$-group of coclass $2$ with $p>2$. Then $G$ has a noninner automorphism 
of order $p$. Moreover this noninner automorphism leaves either $\Fr(G)$ or $Z_{n-4}$
elementwise fixed when $n\geq 7$. 
\end{thm} 

\begin{proof}
First we note that  $n\geq7$ by \cite{Bodna} and \cite{GAP}. Moreover we may assume that $G$ satisfies the conditions of the Remark by \cite[Theorem]{Shabani} and so we use the notation and assumption of the above lemmas. We have   
$\Aut_c(G) \leq \Aut_{N}^N(G)$ since if $\gamma \in \Aut_c(G)$, then $\gamma=i_{g}$
for some $g\in G$ by Lemma \ref{1}(iv) which implies that $g \in Z_2$ and so $\gamma \in \Aut_N(G)$. Therefore 
$\Aut_c(G) \leq \Aut_{N}^N(G)\cap \Inn(G)\leq \Aut^{Z_{2}}(G)\cap\Inn(G)\cong Z_{3}(G)/Z(G).$
Hence by Corollary \ref{6}, if $\alpha^* \in \Inn(G)$, then $\Aut_{N}^N(G)\cap \Inn(G)=\Aut_c(G)\la \alpha^* \ra$. Moreover if $\beta^* \in \Inn(G)$, then $\beta^* \in \Aut_c(G)\la \alpha^* \ra$, which is impossible, by considering the image of $\beta^*$ on $a$. Therefore $\alpha^*$ or $\beta^*$ is noninner. Furthermore by Lemma \ref{5}(iii), (iv), we see that $\alpha^*$ leaves
$\Fr(G)$ and $\beta^*$ leaves $Z_{n-4}$ elementwise fixed, as desired.
\end{proof}

 
\end{document}